\documentclass[11pt]{amsart}

\usepackage{geometry}
\usepackage{color, graphicx}
\usepackage[all]{xy}

\usepackage[T1]{fontenc} 
\usepackage{textcomp}
\usepackage{times}

\swapnumbers
\newtheorem{thm}{Theorem}[section]
\newtheorem{lem}[thm]{Lemma}
\theoremstyle{definition}
\newtheorem{remark}[thm]{Remark}
\newtheorem{se}[thm]{}

\usepackage{amsfonts}
\usepackage{amssymb}

\renewcommand{\bar}{\overline}
\newcommand{\isomto}{\overset{\sim}{\rightarrow}}
\DeclareMathOperator{\Z}{\mathbf{Z}}
\usepackage{bbm}
\DeclareMathOperator{\one}{\mathbbm{1}}

\begin{document}

\date{\today\ (version 1.0)} 
\title{Graph Reconstruction and Quantum Statistical Mechanics}
\author[G.~Cornelissen]{Gunther Cornelissen}
\address{\normalfont Mathematisch Instituut, Universiteit Utrecht, Postbus 80.010, 3508 TA Utrecht, Nederland}
\email{g.cornelissen@uu.nl}
\author[M.~Marcolli]{Matilde Marcolli}
\address{\normalfont Mathematics Department, Mail Code 253-37, Caltech, 1200 E.\ California Blvd.\ Pasadena, CA 91125, USA}
\email{matilde@caltech.edu}

\subjclass[2010]{05C99, 37B10,  46L55, 53C24, 81T75, 82B10 }
\keywords{\normalfont graph, graph reconstruction, quantum statistical mechanics, graph $C^*$-algebra}

\begin{abstract} \noindent We study in how far it is possible to reconstruct a graph from various Banach algebras associated to its universal covering, and extensions thereof to quantum statistical mechanical systems. It turns out that most the boundary operator algebras reconstruct only topological information, but the statistical mechanical point of view allows for complete reconstruction of multigraphs with minimal degree three. 
  \end{abstract}

\maketitle

\section{Introduction}

In this paper, we treat reconstruction problems for graphs from operator algebras. This fits into a broader area of research into reconstruction questions, such as the isospectrality problem in differential geometry \cite{Berard}, and the arithmetic equivalence problem in number theory \cite{Klingen} \cite{CM}.  In recent years, a new philosophy has emerged, stating roughly that mathematical objects that obey some form of measure-theoretic rigidity can be reconstructed up to suitable isomorphism from an associated \emph{quantum statistical mechanical system}, in the sense of noncommutative geometry. A prime example is given by the case of number fields under the dynamics of abelian class field theory \cite{CM}, as in the Bost-Connes system \cite{BC} \cite{HP}.  In this paper, we study a ``baby case'', namely, that of graphs. Ultimately, we hope our work provides a measure-theoretic or functional analytic handle on reconstruction problems for graph invariants from their decks (compare \cite{graphdeck}). 

\begin{se} We first set up basic notations and definitions that allow us to formulate the main theorem. 
By a graph $X$ we mean a finite, unoriented multigraph, allowing loops and multiple edges. We assume throughout that all vertices of $X$ have degree ($=$ valency, i.e., twice the number of loops plus the number of non-loops emanating from the vertex) at least three. In particular, the graph has no sinks. Let $T_X$ denote the universal covering tree, and $\Gamma_X$ the fundamental group. Choose generators $\gamma_1,\dots,\gamma_{g_X}$ for $\Gamma_X$. Let $\Lambda_X$ denote the boundary of $T_X$ (equivalence classes of rays, as a Cantor set). 
\end{se}

\begin{se} We will use dynamical systems in two slightly different forms, one involving groups, and one involving specific generators. We now outline these concepts.

We call a pair $(X,G)$ consisting of a topological space $X$ and a (topological) group $G$ acting on the space by homeomorphisms a \emph{dynamical system}, and say that two such systems $(X,G)$ and $(Y,H)$ are \emph{isomorphic} if there is a homeomorphism $$\Phi \colon X \isomto Y$$ equivariant w.r.t.\ a group isomorphism $\alpha \colon G \isomto H$, i.e., such that $$\Phi(gx)=\alpha(g)\Phi(x)$$ for all $x \in X$ and $g \in G$. 

We say two such systems are \emph{locally isomorphic} if there is a homeomorphism $\Phi \colon X \isomto Y$ and group isomorphisms $\alpha_x \colon G \isomto H$ for any $x \in X$, such that $$\Phi(gx)=\alpha_x(g)\Phi(x)$$ for all $x \in X$ and $g \in G$, and such that $$\alpha_x \colon X \times G \rightarrow  H$$ is locally constant in $x$. Of course, isomorphism implies local isomorphism, but the converse need not be true. 

\end{se}

\begin{se} The second concept depends on a choice of generators for the group. 
If the group $G$ is finitely generated, we can instead consider a \emph{multivariable dynamical system} $(X,\{g_i\})$, consisting of the space $X$ and \emph{given} generators $g_i$ of $G$. We say two such systems $(X,\{g_i\})$ and $(Y,\{h_j\})$ are \emph{conjugate} if there is a homeomorphism $\Phi \colon X \isomto Y$, $\{g_i\}$ and $\{h_j\}$ have the same number $n$ of elements, and there is a permutation $\sigma \in S_n$ such that $$\Phi(g_i x) = h_{\sigma(i)} \Phi(x)$$ for all $x \in X$ and $i=1,\dots,n$. 

We say two such systems are \emph{piecewise conjugate} if there is an open cover $$\{X_\sigma\}_{\sigma \in S_n} \mbox{ of }X$$ such that $$\Phi(g_i x) = h_{\sigma(i)} \Phi(x)$$ for all $x \in X_{\sigma}$. If two systems are conjugate for chosen sets of generators, then they are isomorphic, but the converse need not be true. The relation between local isomorphism and piecewise conjugacy is less clear. (However, in the case we study below, it turns out all these concepts are equivalent.)
\end{se}

\begin{se} We now come to the definitions of the operator algebras that we will consider. 
Let $$A_X:=C(\Lambda_X) \rtimes \Gamma_X$$ denote the full crossed product \emph{boundary operator algebra} generated as $C^*$-algebra by $C(\Lambda_X)$ and group-like elements $\mu_1,\dots,\mu_{g_X}$ corresponding to the generators $\gamma_i$. Let $A_X^{\dagger}$ denote the (non-involutive) subalgebra generated by $C(\Lambda_X)$ and the group elements, but not their $\ast$-companions $\mu_i^*$. 

We assume throughout that $X$ has first Betti number $g_X>1$ (they other cases are not interesting since if $g_X=0$, then $A_X=\mathbf{C}$ and if $g_X=1$, $A_X=\mathbf{C}^2 \times \Z$, since the fundamental group is abelian and the limit set consists of the two fixed points of the fundamental group). Our first main theorem says that for graphs, all these concepts essentially store the same amount of (limited) topological information, namely, the number of independent loops: 
\end{se}

\begin{thm}  \label{m} Let $X$ and $Y$ denote two graphs with $g_X,g_Y>1$, and all of whose vertices have degree at least three. Then the following are equivalent: 
\begin{enumerate}
\item[\textup{(a)}] $A_X \cong A_Y$ are strictly isomorphic as $C^*$-algebras; \label{a}
\item[\textup{(b)}]  $A_X \cong A_Y$ are stably isomorphic as $C^*$-algebras; \label{b}
\item[\textup{(c)}]  $A_X \cong A_Y$ are strongly Morita equivalent; \label{c}
\item[\textup{(d)}]  $A_X^{\dagger} \cong A_Y^{\dagger}$ are strictly isomorphic as Banach algebras (for any choice of generators of $\Gamma_X$ and $\Gamma_Y$); \label{d}
\item[\textup{(e)}]  The multivariable dynamical systems $(\Lambda_X,\{\gamma_{X,i}\})$ and $(\Lambda_Y,\{\gamma_{Y,i}\})$ are piecewise conjugate (for any choice of generators of $\Gamma_X$ and $\Gamma_Y$); \label{e}
\item[\textup{(f)}]  The multivariable dynamical systems $(\Lambda_X,\{\gamma_{X,i}\})$ and $(\Lambda_Y,\{\gamma_{Y,i}\})$ are conjugate (for any choice of generators of $\Gamma_X$ and $\Gamma_Y$); \label{f}
\item[\textup{(g)}]  The dynamical systems $(\Lambda_X,\Gamma_X)$ and $(\Lambda_Y,\Gamma_Y)$ are locally isomorphic; \label{f1}
\item[\textup{(h)}]  The dynamical systems $(\Lambda_X,\Gamma_X)$ and $(\Lambda_Y,\Gamma_Y)$ are isomorphic; \label{f2}
\item[\textup{(i)}]  $g_X = g_Y$. \label{g}
\end{enumerate}
\end{thm} 

The proof, given in the next section, combines work of Robertson on boundary operator algebra $K$-theory \cite{R} with a result of Davidson and Katsoulis on the relations between piecewise conjugacy of dynamical systems and operator algebras \cite{DK}, and the theory of hyperbolic groups. 

By stark contrast, if we allow for further noncommutative dynamics, implemented as a \emph{quantum statistical mechanical system} by a one-parameter group of automorphisms of the dagger algebra, the graph is uniquely determined up to isomorphism.

\begin{thm} \label{mm} 
Let $X$ and $Y$ denote two graphs with $g_X,g_Y>1$, and all of whose vertices have degree at least three.  Then one can define a one-parameter subgroup of automorphisms (``time evolution'') $$ \sigma_X \colon \mathbf{R} \rightarrow \mathrm{Aut}(A_X^\dagger) $$ as follows: choose a base point $x_0$ in the universal covering tree $T_X$ of $X$, and set 
\begin{equation}
\sigma_t\left( \sum_{\gamma} f_\gamma(\xi) \mu_\gamma\right) = \sum_{\gamma} \exp\left(it \lim\limits_{{x_3 \in T_X} \atop{x_3\to \xi}} d(x_0,x_3)-d(\gamma^{-1}x_0,x_3)\right)  f_\gamma(\xi) \mu_\gamma
\end{equation}
on elements $\sum f_\gamma \mu_\gamma$ of the algebraic crossed product algebra (with $f_\gamma \in C(\Lambda_X)$). 
Then the quantum statistical mechanical systems $$(A_X^{\dagger},\sigma_X) \cong (A_Y^{\dagger},\sigma_Y)$$ are isomorphic if and only if $X$ and $Y$ are isomorphic as graphs. 
\end{thm}

The time evolution is defined in terms of Busemann functions as in the work of Coornaert and Lott \cite{Lott}, and the proof is based on a classification of KMS-states for such algebras by Kumjian and Renault \cite{Renault}, combined with  measure-theoretic rigidity for trees from Paulin and Hersonsky \cite{Paulin}. The choice of a base point is irrelevant and relates to the the fact that Patterson-Sullivan measures naturally form a conformal density, rather than a measure.

\begin{remark}
This paper studies reconstruction of graphs from various structures related to the boundary operator algebra $A_X$. Similar questions can be asked for various Cuntz-Krieger algebras that are naturally associated to a graph, such as the one associated to the vertex adjancency matrix of the graph (also known as the \emph{graph $C^*$-algebra}), or the Cuntz-Krieger algebra associated to the Bass-Hashimoto edge adjancency operator. We will comment on this problem at the appropriate place. 
\end{remark}

\begin{remark}
The results of this paper carry over one-to-one to the case of Riemann surfaces uniformized by Schottky groups; we refrain from presenting details. 
\end{remark}

\section{Proof of Theorem \ref{m}}

In this first section, we consider the various natural ``isomorphism'' conditions on the boundary operator algebras associated to a graph, and discuss in how far they (do not) determine the graph. We keep the notations from the introduction. A schematic representation of the structure of the proof is given below: 
$$
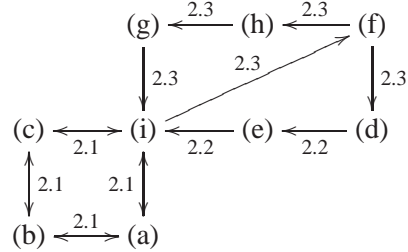
\xymatrix{ 
& \text{(g)} \ar@{->}[d]^{\text{\ref{5}}} \ar@{<-}[r]^{\text{\ref{5}}}  &  \text{(h)} \ar@{<-}[r]^{\text{\ref{5}}} & \text{(f)}  \ar@{->}[d]^{\text{\ref{5}}} \ar@{<-}[dll]_{\text{\ref{5}}} \\ 
\text{(c)} \ar@{<->}[r]_{\text{\ref{1}}} \ar@{<->}[d]^{\text{\ref{1}}}   & \text{(i)} \ar@{<-}[r]_{\text{\ref{3}}} \ar@{<->}[d]_{\text{\ref{1}}}   &  \text{(e)} \ar@{<-}[r]_{\text{\ref{3}}} & \text{(d)} \\
\text{(b)} \ar@{<->}[r]^{\text{\ref{1}}}  &  \text{(a)} &  &
     }$$
     
\begin{se}[$C^*$-algebra isomorphism] \label{1}
The boundary $\Lambda_X$ of the tree $T_X$ is the set of equivalence classes of half-rays in $T_X$, with equivalent rays being equal up to finitely many edges. The full crossed product algebra $C(\Lambda_X) \rtimes \Gamma_X$ is the universal $C^*$-algebra generated by the commutative $C^*$-algebra $C(\Lambda_X)$ and the image of a unitary
representation $\pi$ of $\Gamma$, satisfying the covariance relation $$ \pi(\gamma)f\pi(\gamma)^{-1}=f \circ \gamma^{-1}$$
for $f \in C(\Lambda_X), \gamma \in \Gamma$ and $x \in \Lambda_X.$  It is also strictly isomorphic as $C^*$-algebra to the Cuntz-Krieger algebra of the edge adjacency matrix (=Bass-Hashimoto operator, cf.\ \cite{CLM}) of a graph with one vertex and $g_X$ loops, i.e., the Cuntz-Krieger algebra corresponding to the matrix $$\bigl( \begin{smallmatrix} 
  \one & \one-\mathbf{1}\\
  \one-\mathbf{1} & \one 
\end{smallmatrix} \bigr)$$
where $\one$ is a square $g_X$-size matrix all of whose entries are one, and $\mathbf{1}$ is the $g_X \times g_X$-unit matrix cf.\ \cite{R}.  Since this algebra fits into the Kirchberg-Phillips classification, it is determined by its $K_0$-group and the position of the unit (cf.\ R{\o}rdam \cite{Rordam}), which are easily computed for this matrix to be $$K_0 = \Z^{g_X} \times \Z/(g_X-1),$$ and the class of the unit being a generator for the finite cyclic part. Thus, both strict isomorphism and Morita equivalence of such algebras is equivalent to equality of the Betti numbers of the graphs: $g_X=g_Y$. Also, since the algebras are separable, strong Morita equivalence is the same as stable isomorphism (i.e., isomorphism after tensoring with the algebra of compact operators). This proves the equivalence of (a), (b), (c) and (i) of Theorem \ref{m}. 
\end{se}

\begin{se}[{Non-involutive boundary algebra isomorphism}] \label{3}
We now switch to the non-involutive subalgebra $A_X^{\dagger}$ of $A_X$. As a subalgebra of the universal crossed product $C^*$-algebra $A_X$, the algebra $A_X^{\dagger}$ can be described as the \emph{semicrossed product} algebra of the chosen generators of the free group $\Gamma_X$ acting on the (totally discontinous space) $\Lambda_X$. This algebra is universal for the equivariance relation given by the group action with the group generator elements of the free group being contractions or isometries (which is the same, by a suitable dilation theorem, cf.\ \cite{DK}). Then one of the main results (Thm.\ 9.2) in \cite{DK} states that $A^\dagger_X \cong A^\dagger_Y$ implies a \emph{piecewise topological conjugacy} of the dynamical systems $(\Lambda_X, \{\gamma_{X,i}\})$ and $(\Lambda_Y, \{\gamma_{Y,j}\})$ (hence in our situation, it says that (d) implies (e)). 

Piecewise conjugacy implies that the ranks of the group $\Gamma_X$ and $\Gamma_Y$ are equal, whence $g_X=g_Y$, so that (e) implies (i). 
\end{se}

\begin{se}[Betti number condition] \label{5}
Now we prove that two such dynamical systems are conjugate (for any choice of generators of $\Gamma_X$ and $\Gamma_Y$) precisely if $g_X=g_Y$. We use the fact that the group $\Gamma_X$ is hyperbolic (since we assume $g_X>1$), and the action of $\Gamma_X$ on $T_X$ is properly discontinuous with compact quotient. These conditions imply that there is a canonical $\Gamma_X$-equivariant homeomorphism $$\varphi_X \colon \partial \mathrm{Cay}(\Gamma_X,\{\gamma_{X,i}\}) \isomto \Lambda_X,$$ cf.\ \cite{CDP}, Thm.\ 4.1, where $\partial \mathrm{Cay}(\Gamma_X,\{\gamma_{X,i}\})$ is the boundary of a Cayley graph $$\mathrm{Cay}(\Gamma_X,\{\gamma_{X,i}\})$$ for the free group $\Gamma_X$ of rank $g_X$ for the symmetrization of the chosen set of generators $\{\gamma_{X,i}\}$.  Namely, the map $$\varphi_X \colon \mathrm{Cay}(\Gamma_X,\{\gamma_{X,i}\}) \rightarrow T_X \colon \gamma \mapsto \gamma x_0$$ is a quasi-isometry for the usual tree distances, so it will induce a homeomorphism of boundaries. Observe that a different choice of generators also changes the metric to a Lipschitz-equivalent one, inducing the same homeomorphism in the boundary. 

Since $\varphi_X$ is equivariant w.r.t.\ the action by left multiplication of $\Gamma_X$, the resulting boundary map is indeed equivariant.  Now suppose that $g_X=g_Y$. Then we can choose a group isomorphism $$\alpha \colon \Gamma_X \cong \Gamma_Y,$$ with $\alpha(\gamma_{X,i})=\gamma_{Y,i}$ (so extending a bijection between the given sets of generators). then $\alpha$ induces a corresponding equivariant boundary map, and this implies that we have a composed homeomorphism 
$$ \Phi \colon \Lambda_X \xrightarrow{\varphi_X^{-1}} \partial \mathrm{Cay}(\Gamma_X,\{\gamma_{X,i}\}) \xrightarrow{\alpha^*} \partial \mathrm{Cay}(\Gamma_Y,\{\gamma_{Y,i}\}) \xrightarrow{\varphi_Y} \Lambda_Y, $$
which intertwines the actions of $\Gamma_X$ and $\Gamma_Y$ with $$\Phi(\gamma_{X,i} x) = \alpha(\gamma_{Y,i}) \Phi(x).$$ This is precisely a conjugacy of the multivariable dynamical systems, and proves that (i) implies (f). Of course, (f) also implies (e). Since (f) implies (h), which implies (g), which it its turn implies (i). We now prove that (f) implies (d): if the conjugacy is implemented by a homeomorphism $\Phi \colon \Lambda_X \isomto \Lambda_Y$ with $\Phi(\gamma_{X,i} x) = \gamma_{Y,\sigma(i)} \Phi(x),$ for some permutation $\sigma \in S_{g_X}$, then we define the algebra isomorphism $\varphi \colon A_X^{\dagger} \isomto A_Y^{\dagger}$ on elements $\sum f_i \mu_{\gamma_{X,i}}$ of the algebraic semicrossed product subalgebra by 
$$ \varphi\left( \sum f_i \mu_{\gamma_{X,i}}\right) := \sum f_i \circ \Phi^{-1} \mu_{\gamma_{Y,\sigma(i)}}. $$
This finishes the proof of the theorem. \hfill $\Box$
\end{se}

\begin{remark} We have proven that in our system, local isomorphism (g) implies isomorphism (h), but in a very indirect way. One can also prove this directly, using only topological properties of the system, namely, the fact that the complement of the fixed point set is dense in the space. Since this argument might be of independent interest, we present it here. 
\end{remark} 

\begin{lem}
Suppose that $(\Lambda_X,\Gamma_X)$ and $(\Lambda_Y,\Gamma_Y)$ are two locally isomorphic dynamical systems. Denote the total fixed point set by 
$$ F_X:=\bigcup_{\gamma \in \Gamma_X-\{1\}} \mathrm{Fix}(\gamma), $$
with $\mathrm{Fix}(\gamma)$ is the set of fixed points of $\gamma$. Then if the complement of the fixed point set $\Lambda_X \setminus F_X$ is dense in $\Lambda_X$, the systems $(\Lambda_X,\Gamma_X)$ and $(\Lambda_Y,\Gamma_Y)$ are conjugate. \end{lem} 

\begin{proof} Suppose given the locally constant group isomorphisms $\alpha_x$ that intertwine the actions. 
The equivariance implies on the one hand
$$ \Phi(\gamma_1 \gamma_2 x) = \alpha_x(\gamma_1 \gamma_2) \Phi(x) = \alpha_x(\gamma_1) \alpha_x(\gamma_2) \Phi(x) =
\alpha_x(\gamma_1)  \Phi(\gamma_2 x), $$
and on the other hand
$$  \Phi(\gamma_1 \gamma_2 x) = \Phi(\gamma_1 (\gamma_2 x)) = \alpha_{\gamma_2 x}(\gamma_1) \Phi(\gamma_2 x). $$
We conclude that for all $x \in \Lambda_X$ and for all $\gamma_1, \gamma_2 \in \Gamma_X$, we have 
\begin{equation} \label{nofix} [\alpha_{\gamma_2 x}(\gamma_1)^{-1} \alpha_x(\gamma_1)] \Phi(\gamma_2 x) = \Phi(\gamma_2 x). \end{equation}
Now choose $x$ such that 
\begin{equation} \label{suchx} x \in \Lambda_X \setminus F_X, \end{equation}
where  $F_X$ is the fixed point set. This implies that 
$\gamma_2 x \in \Lambda_X \setminus F_X, $
and then in turn we find
$ \Phi(\gamma_2 x) \in \Lambda_Y \setminus F_Y, $
since if $\gamma x = x$ for some $\gamma \in \Gamma_X$, then $\Phi(\gamma x) = \Phi(x)$, and equivariance gives $\alpha_x(\gamma) \Phi(x) = \Phi(x)$, i.e., there is an element (namely, $\alpha_x(\gamma)$) in $\Gamma_Y$ that fixes $\Phi(x)$. 
We conclude from (\ref{nofix}) that for such $x$ with (\ref{suchx}), it holds true that 
$$ \alpha_{\gamma_2 x} = \alpha_x \mbox{ for all } \gamma_2 \in \Gamma_X,$$
i.e., $\alpha_x$ is locally constant on orbits of the non-fixed point set. Since $\alpha_x$ is continuous in $x$, and constant on the dense subset $\Lambda_X \setminus F_X$ of its domain, we conclude that $\alpha_x$ is constant, independent of $x$. \end{proof}

The Lemma applies in our case, since the set $\Lambda_X \setminus F_X$ is dense in $\Lambda_X$: indeed, if $x \in F_X$ is fixed by some $\gamma \in \Gamma_X$, then there exists $y \in \Lambda_X \setminus F_X$, but the $\gamma$-orbit of $y$ accumulates at $x$ (being one of the fixed points of $\gamma$). Hence $x$ belongs to the closure of $\Lambda_X \setminus F_X$. This proves that the closure of $F_X$ belongs to the closure of $\Lambda_X \setminus F_X$, and the proof is finished since the closure of $F_X$ is all of $\Lambda_X$.

\begin{remark}[Other $C^*$-algebras of graphs] Alternatively, one might want to study directly the relation between (stable) isomorphism of the Cuntz-Krieger algebras associated to the vertex adjacency matrix of two (oriented) graphs. Again, the results of R{\o}rdam et.\ al.\ imply that the $K_0$-group of the algebra determine such isomorphism, and those are given by $$\Z^{|VX|}/\mathrm{im}(1-A_X)$$ for a graph $X$ with vertex set $VX$ and adjacency operator $A_X$. Recently, S{\o}rensen \cite{Sorensen} characterized isomorphism of such algebras purely in terms of the existence of a sequence of elementary operations that transforms one graph into the other, so that one now has a \emph{graph theoretical characterization of stable isomorphism.} It would be interesting to have such a description also for strict isomorphism. 
\end{remark}

\begin{remark}
The situation becomes much simpler if one instead studies (stable) isomorphism of the Cuntz-Krieger algebras associated to the edge adjancency matrix (Bass-Hashimoto operator) of the graph. This operator $\mathbb{T}_X$ is defined as follows (compare with \cite{Bass} or \cite{CLM}). Choose a random orientation of the edged $\{e_1,\dots,e_n\}$ of $X$ and let $\{e_1,\dots,e_n,\bar{e}_1,\dots,\bar{e}_n\}$ denote the set of all edges of $X$ and their inverses. The operator $\mathbb{T}_X$ is the $2n \times 2n$-matrix given by $T(e,f)=1$ if and only if the terminal vertex of $e$ is the initial vertex of $f$, but $f \neq \bar e$. Otherwise, $T(e,f)=0$.  This operator is very natural for counting non-backtracking paths in the graph (relating to the Ihara zeta function). 

In \cite{CLM}, one finds a geometric study of this classification problem, and the answer is that stable isomorphism of such algebras for graphs $X$ and $Y$ is the same as equality of Betti numbers $g_X=g_Y$, and strict isomorphism means that we also have and equality $$(|EX|,g_X-1)=(|EY|,g_Y-1)$$ of greatest common divisors of the Euler characteristic of the graphs with their number of edges. 
\end{remark}

\section{Proof of Theorem \ref{mm}}

\begin{se}[{Quantum statistical mechanical system and Busemann function}]
The dynamics $\sigma_X=\sigma_{X,x_0}$ on $A_X^{\dagger}$ is given in terms of the Busemann function (compare with Coornaert \cite{CoornaertPS}, or Lott \cite{Lott} for the similar case of Kleinean groups), defined for $x_1,x_2 \in T_X$ and $x \in \Lambda_X$ by 
$$ B(x_1,x_2,x) =\lim_{{x_3 \in T_X} \atop{x_3\to x}} d(x_1,x_3)-d(x_2,x_3), $$
with $d$ the tree distance. The automorphisms $\sigma_t$ are defined on elements of the algebraic crossed product, i.e., finite sums of the form $$ f = \sum_{\gamma \in \Gamma_X} f_\gamma \mu_\gamma$$ with $f_\gamma \in C(\Lambda_X)$, and then extended to the full $C^*$-algebra (full = reduced in this case). Choose a ``base point'' $x_0 \in T_X$. Then one defines $\sigma \colon \mathbf{R} \rightarrow \mathrm{Aut}(A_X^{\dagger}) \colon t \mapsto \sigma_t$ by 
\begin{equation}\label{sigmaB}
\sigma_{X,x_0,t}( \sum_{\gamma} f_\gamma \mu_\gamma) = \sum_{\gamma} e^{it B(x_0,\gamma ^{-1} x_0,-)}  f_\gamma \mu_\gamma.
\end{equation}
This indeed defines a strongly continuous one-parameter subgroup of automorphisms of $A_X$ (respecting $A_X^{\dagger}$); the proof is the same as that of Proposition 4.2 in \cite{Lott}. Actually, it defines a field of such automorphisms groups, by varying the base point $x_0$, but when no confusion can arise, we will leave out the base point from the notation and simply write $\sigma_X = \sigma_{X,x_0}$.

One calls $(A^{\dagger}_X, \sigma_X)$ the \emph{quantum statistical mechanical system} associated to the graph $X$ and $\sigma_X$ the \emph{time evolution} of the system (cf.\ \cite{BR} for background on this terminology). An \emph{isomorphism} of two such systems $(A^{\dagger}_X,\sigma_X)$ and $(A^{\dagger}_Y,\sigma_Y)$ is a Banach algebra isomorphism $\varphi \colon A^{\dagger}_X \isomto A^{\dagger}_Y$ that intertwines the time evolutions: $ \varphi \circ \sigma_X = \sigma_Y \circ \varphi.$  
\end{se} 

\begin{remark} We have defined the system in terms of the non-involutive algebra $A_X^{\dagger}$. The above time evolution extends easily to the full $C^*$-algebra $A_X$ and is seen to respect the subalgebra $A_X^{\dagger}$, and traditionally, a quantum statistical mechanical systems consists of such a $C^*$-algebra $A_X$ with time evolution $\sigma_X$. One may also describe our notion of isomorphism as an isomorphism of these ``involutive'' systems $(A_X,\sigma_X) \isomto (A_Y,\sigma_Y)$ that furthermore maps $A_X^{\dagger}$ to $A_Y^{\dagger}$. 
\end{remark}

Continuing with the proof of Theorem \ref{mm}, suppose given an isomorphism
\begin{equation}\label{alphacross}
\varphi \colon (A_X^{\dagger}, \sigma_X) \stackrel{\simeq}{\longrightarrow} (A_Y^{\dagger}, \sigma_Y)
\end{equation}
of quantum statistical mechanical systems. The first thing to observe is that such an isomorphism induces (by pull-back) a homeomorphism between the Choquet-complexes of KMS$_\beta$-states of the systems $(A_X,\sigma_X)$ and $(A_Y,\sigma_Y)$ (for any ``temperature'' $\beta$), cf.\ \cite{CM}, Lemma 1.7. 

The second observation is the fact that the system $(A_X,\sigma_X)$ has a \emph{unique} $\beta$ for which the space of $\mathrm{KMS}_\beta$-states is non-empty. The reason is that there is a correspondence between $\mathrm{KMS}_\beta$-states and conformal measures (compare with \cite{Lott} 5.11 for the case of Kleinean groups). Namely,  
since the dynamical system $(\Lambda_X,\Gamma_X)$ is the same as the full shift space on chosen generators for the free group, and left shift is an expansive map on this space, we can use the work of Kumjian and Renault \cite{Renault}: the $C^*$-algebra $A_X$ is the reduced $C^*$-algebra of the transformation groupoid $\Lambda_X \times \Gamma_X$. To the base point, we associate a one-cocyle $c_{x_0} \in Z^1(\Lambda_X \times \Gamma_X,\mathbf{R})$ by 
$$ c_{x_0}(\xi,\gamma):=B(x_0, x_0 \gamma^{-1}, \xi). $$
Now if $\gamma \xi = \xi$ and $c_{x_0}(\xi, \gamma)=0$, then $c_{x_0}^{-1}(0)$ is a principal subgroupoid of the transformation groupoid. Indeed, since $\gamma$ is a hyperbolic isometry of the tree, the classification of Tits \cite{Tits} implies that $\gamma$ acts by a translation with amplitude $\mathrm{am}(\gamma)$ along the double sided infinite path in the tree $T_X$ that connects the two fixed points of $\gamma$ on the boundary, called the \emph{axis} of $\gamma$. One of these fixed points is $\xi$, and if we let $x_3 \in T_X$ approach this point along the axis of $\gamma$, it is obvious from the definition of the Busemann function that $$B(x_0,x_0\gamma^{-1},\xi) = \pm \mathrm{am}(\gamma).$$ Hence if this cocycle is zero, we find $\mathrm{am}(\gamma)=0$, and hence $\gamma=1$. 	
This means we can indeed apply the result of Kumjian and Renault, and conclude that the $\mathrm{KMS}_\beta$-states correspond to \emph{conformal measures} of dimension $\beta$ on $\Lambda_X$, i.e., probability measures $\mu$ on $\Lambda_X$ that satisfy the scaling condition
$$ \mu(\gamma \xi) = e^{\beta B(x_0,x_0 \gamma^{-1},\xi)} \mu(\xi). $$
Now Coornaert (\cite{CoornaertPS}, Section 8) has classified such conformal measures: they only occur for $\beta=\delta(\Lambda_X)$ the critical exponent of the Poincar\'e series of $\Gamma$, and they are exactly the Patterson-Sullivan measures $\mu_{x_0}$ for varying base points $x_0$ (and these are normalized Hausdorff measures on $\Lambda_X$). 

The conclusion is that $(A_X,\sigma_{X,x_0})$ has a unique $\beta$ for which there exists a $\mathrm{KMS}_{\beta}$-state. Actually, from the correspondence between states and measures, this $\mathrm{KMS}_\beta$-state is given by 
$$ \tau_{X,x_0} ( \sum_{\gamma} f_\gamma \mu_\gamma) = \int_{\Lambda_X} f_{1} d\mu_{X,x_0}. $$ 

Now to finish the proof, recall from the previous section that the $\dagger$-algebra map $\varphi$ induces a homeomorphism $\Phi \colon \Lambda_X \rightarrow \Lambda_Y$ that is equivariant w.r.t.\ a group isomorphism $\alpha \colon \Gamma_X \rightarrow \Gamma_Y$. Since the unique normalized $\mathrm{KMS}_\beta$-state for $Y$ (with $\beta$ the Hausdorff dimension of $\Lambda_Y$, and chosen base point $y_0$) pulls back via $\varphi$ to a $\mathrm{KMS}_\beta$-state on $X$, and this (normalized) state is unique and given by $\tau_{X,x_0}$ for some base point $x_0$, and we find
$$ \varphi^* \tau_{Y,y_0} = \tau_{X,x_0},$$
which translates into 
$$ \int_{\Lambda_Y} f_1 \circ \Phi^{-1} d\mu_{Y,y_0} = \int_{\Lambda_X} f_1 d\mu_{X,x_0}, $$
for any function $f_1 \in C(\Lambda_X)$. 
This means that $\Phi$ is a measure-preserving (in particular, absolutely continuous) homeomorphism of the limit sets in their respective Patterson-Sullivan measures based at the points $x_0$, respectively $y_0$, that is equivariant for the given group isomorphism. Actually, since different Patterson-Sullivan measures for different choices of base points are mutually absolutely continuous, we find that for any choice of such measures, $\Phi$ is absolutely continuous. 

Now the rigidity theorem for hyperbolic group actions implies that $\Phi$ is M\"obius on the boundary (meaning that it respects the generalized ``cross ratio'' on the boundary \cite{Paulin}, Theorem A). Hence it extends to an isometry between the trees $T_X$ and $T_Y$ that intertwines the respective group actions (this extension is defined by coordinatizing a point in the tree as the unique intersection of three half lines ending at points on the boundary, and defining the image point as the point coordinatized in this way by the three image points under the boundary map of the original three boundary points; compare \cite{CoornaertCRAS}, \cite{JWgraphs}). Hence finally, it induces a graph isomorphism $$\Phi \colon X \isomto \Gamma_X \backslash T_X \isomto \Gamma_Y \backslash T_Y \rightarrow Y. $$ This finishes the proof of the theorem.  

\begin{remark}
We do not know whether one can relax the conditions in the theorem by leaving out the dagger algebra; is it true that an isomorphism of quantum statistical mechanical systems $(A_X,\sigma_X) \isomto (A_Y,\sigma_Y)$ suffices to conclude that the graphs $X$ and $Y$ are isomorphic? 
\end{remark}

\begin{remark}
Analogously, one can try to define a time-evolution on the Cuntz-Krieger algebra (or a non-involutive subalgebra of it) of the vertex- or edge-adjancency matrix of the graph, and study whether the resulting quantum statistical mechanical system reconstructs the graph. We have not completed such a study. 
\end{remark}

\end{document}